\documentclass[12pt]{amsart}
\usepackage{amssymb,amscd}
\usepackage[all,cmtip]{xy}
\usepackage{epic}
\usepackage{url}
\usepackage[english]{babel}

\numberwithin{equation}{section}

\newtheorem{theorem}{Theorem}[section]
\newtheorem{lemma}[theorem]{Lemma}
\newtheorem{proposition}[theorem]{Proposition}

\newtheorem{definition-lemma}[theorem]{Definition-Lemma}
\theoremstyle{definition}
\newtheorem{definition}[theorem]{Definition}

\newtheorem{example}[theorem]{Example}

\theoremstyle{remark}
\newtheorem{remark}[theorem]{Remark}

\def\FF{{\mathbb F}}

\def\PP{{\mathbb P}}
\def\QQ{{\mathbb Q}}

\def\ZZ{{\mathbb Z}}

\def\cO{{\mathcal O}}

\newcommand\la{\langle}
\newcommand\ra{\rangle}

\begin{document}
\title
{Optimal Curves over Finite Fields with Discriminant -19}

\author{E.\ Alekseenko , S.\ Aleshnikov , N.\ Markin,  A.\ Zaytsev}

\maketitle
\begin{abstract}
In this work we study the properties of maximal and minimal curves of genus $3$ over finite fields with discriminant $-19$. 
We prove that any such curve can be given by an explicit equation of certain form (see Theorem \ref{equations}).
Using these equations we obtain a table of maximal and minimal curves over prime finite fields with discriminant $-19$ of cardinality up to $997$.
We also show that existence of a maximal curve implies that there is no minimal curve and vice versa.

\end{abstract}

\begin{section}{Introduction}

The number of rational points of an irreducible non-singular projective curve $C/\FF_q$ of genus $g$ satisfies the Hasse-Weil-Serre bound:
$$|\#C(\FF_q)-q-1|\leq g[2\sqrt{q}].$$

In case of equality, that is,   $\#C(\FF_{q})=q+1\pm g[2\sqrt{q}]$
the curve is called {\emph{optimal  over $\FF_q$}}. When $\#C(\FF_{q})=q+1- g[2\sqrt{q}]$ it is called a {\emph{maximal over $\FF_q$}} curve and when 
$\#C(\FF_{q})=q+1- g[2\sqrt{q}]$  it is called a {\emph{minimal over $\FF_q$}} curve.

Let $C$ be an optimal  curve  of genus $g$ over $\FF_{q}$. 
Then the  Frobenius  endomorphism induces a homomorphism 
$$
{\rm F}:{\rm T}_l {\rm Jac}(C)\otimes_{\ZZ}{\QQ}\rightarrow {\rm T}_l {\rm Jac}(C)\otimes{_\ZZ}{\QQ},
$$ 
where ${\rm T}_l {\rm Jac}(C)$ is the projective limit ${\rm \underleftarrow{lim} Jac}(C)[l^{n}]$.  
Moreover, if the characteristic polynomial of ${\rm Jac}(C)$  splits                            
$$
{\rm P}_{{\rm Jac}(C)}(T)=\prod ^{2g}_{i=1}(T-\alpha_{i}),
$$ then the number of rational points on $C$ equals to 
$$
\#C(\FF_q)=q+1-\sum ^{g}_{i=1}\alpha_{i}=q+1-\sum ^{g}_{i=1}(\alpha_i+\overline{\alpha}_i),
$$ 
with $\alpha_{i+g}=\overline{\alpha}_i$. The eigenvalues of the endomorphism Frobenius ${\rm F}$ have following property: ${\alpha_{i}}+ {\bar\alpha_{i}= -{[2\sqrt{q}]}}$ when $C$ is  maximal  and ${\alpha_{i}}+ {\bar\alpha_{i}= {[2\sqrt{q}]}}$ when $C$ is minimal. Therefore, if a curve is optimal, then $L$-polynomial of this curve is 
$$
L(t)=\prod ^{2g}_{i=1}(1 - \alpha_it)=\prod^{g}_{i=1}(1 \mp [2\sqrt{q}]t +qt^{2}),
$$  
where the minus sign applies to the minimal case and the plus sign
to the maximal case.
Then the theory of Honda-Tate shows that the Jacobian $\rm Jac(C)$ of a 
maximal 
curve $C$ is isogenous to a product of copies of a maximal
elliptic curve, that is  ${\rm Jac}(C)\sim E^{g}$, where $E$ is a maximal 
elliptic curve over a finite field $\FF_q$.  The similar isogeny occurs if the curve
$C$ and the elliptic curve $E$ are minimal.
 The isogeny class of $E$ over a finite field $\FF_q$ is characterized  by the characteristic polynomial of the Frobenius endomorphism of $E$. 

The definition of the discriminant of a finite field is recalled.
\begin{definition}
For a finite field $\FF_{q}$, the number $[2\sqrt{q}]^2-4q$ is called 
the \emph{discriminant} of a finite field $\FF_q$.
\end{definition}

Now, we consider the equivalence between the category of ordinary abelian varieties ${\rm Jac}(C)$ over $\FF_q$ which are isogenous  to $ E^{g}$ (hence $E$ is ordinary)  and the category of $R$-modules, where $R$ is the ring define by the Frobenius endomorphism of $E$.  In our case let $C$ be a smooth irreducible projective algebraic curve over  $\FF_q$ with discriminant $-19$. Therefore $R = \mathcal O_K$, where $K = \mathbb Q(\sqrt{-19})$. Let ${\rm Jac}(C)$ be the principal polarized Jacobian variety of $C$ with theta-divisor $\theta$. By the Torelli Theorem, the curve $C$ is completely defined by $({\rm Jac}(C),\theta)$, up to a unique isomorphism over an algebraic closure of $\FF_q$. Consider the Hermitian module $(\cO^{g}_{K};h)$, where $\cO^{g}_{K}$ is a $\cO_{K}$-module, and $h:\cO^{g}_{K}\times\cO^{g}_{K}\rightarrow \cO_K$ is a Hermitian form. The equivalence of categories is defined by the functor $\mathcal F: {\rm Jac}(C)\longrightarrow {\rm Hom}(E,{\rm Jac}(C))$ and its inverse  $\mathcal V:\cO^{g}_{K}\longrightarrow \cO^{g}_{K} \otimes _{\cO_{K}}E$. Under this equivalence the principal polarisation of Jacobian ${\rm Jac}(C)$ corresponds to an irreducible Hermitian $\cO_K$-form $h$. Therefore we can  use the classification of unimodular irreducible Hermitian forms in order to study the isomorphism classes of $\rm Jac(C).$ For detailed description of this equivalence of categories see the Appendix by J.-P. Serre in \cite{KL}.

Deligne's Theorem \cite{PD} yields that the number of isomorphism classes of abelian varieties isogenous to $A$ equals the number of isomorphism classes of $R$-modules, which may be embedded as lattices in the $K$-vector space $K^{g}$, where $K={\rm Quot}(R)$. Since in our case there exists one isomorphism class of  such $R$-modules, then there exists a unique isomorphism class of abelian varieties. Therefore Deligne's Theorem together with \ref{elliptic} show that  ${\rm Jac}(C)$ is actually isomorphic to $E^{g}$.

The main result of this paper is putting this theory to practical use. We give a characterization of isomorphism classes of optimal curves of genus $3$ over finite fields with discriminant $-19$ in such a way that we are able to give an explicit description of all such curves. In particular, we produce maximal and minimal curves of genus $3$ over prime finite fields with discriminant $-19$ of cardinality up to $997$.

We also would like to mentioned recent article of Christophe Ritzenthaler  "Explicit computations of Serre's obstruction for genus 3 curves and application to optimal curves " (see \cite{CR}). 
Beside of all other interesting results of this article we would like to point out the method of detecting maximal and minimal curves and the model of a modular curve which can be reduced to optimal curve over finite field with discriminant $-19$. 

\end{section}

\begin{section}{Optimal Curves of Genus $1$ and $2$}

\subsection{Optimal Elliptic Curves}
In this subsection we explore optimal elliptic curves over $\FF_{q}$ and produce 
concrete calculations for the finite fields $\FF_{q}$ of the discriminant $-19$ and  $q \le 1000$.

The endomorphism ring ${\rm End}(E)$ of an elliptic curve $E$ is the set of all isogenies $\phi:E({\overline\FF}_q)\rightarrow E({\overline\FF}_q)$, with multiplication corresponding to composition. If a curve $E$ has complex multiplication, then by Deuring's theory \cite{WW} the endomorphism ring ${\rm End}(E)$ is an  order in the imaginary quadratic field $K=\QQ(\sqrt{d})$. The theory of complex multiplication and Deuring's lifting theory 
give us the following: given a quadratic field $K$, the number of isomorphism classes of elliptic curves over $\FF_q$ whose endomorphism rings are isomorphic to the maximal order $\cO_K$ is equal the number of ideal classes $h_K$ of  $K$. 

\begin{proposition}\label{elliptic}
Let $\FF_q$ be a finite field with discriminant $-19$. 
There exist exactly two $\FF_q$-isomorphism classes of optimal elliptic curves $E$
over $\FF_{q}$, namely the class of maximal and the class of minimal elliptic curves over $\FF_q$.
\end{proposition}

\begin{proof}

Deuring's Theorem provides the existence of maximal and minimal elliptic curves over a finite field with discriminant $-19$. Let $E$ be such a curve. Then $\rm End_{\FF_q}(E)$ contains the ring $\mathbb Z[x]/(x^2+mx +q)$, where $m = \pm
[2 \sqrt{q}]$ is the trace of the Frobenius endomorphism of $E$. Therefore we have ${\rm End_{\FF_q}(E)} \cong \cO_K \cong \mathbb Z[x]/(x^2+mx +q)$, where $K$ is the imaginary quadratic field $\QQ(\sqrt{-19})$ with discriminant $-19$.
We have $\cO_K = \ZZ[\frac{-19+\sqrt{-19}}{2}]$ and the Minkovski bound is $B_K \approx 2,77$. Then any non-principal ideal class must be representable by an ideal of norm $\leq 2,77$. We verify that $2\cO_K$ is a principal ideal to conclude that $h_K = 1$. 
From the class number and the mass formula (see \ref{GG})
it follows that there exists a unique class of isomorphic elliptic curves over $ \FF_q$.
\end{proof}

\begin{remark}
Alternatively, we can find the number of $\FF_q$-isomorphism classes of elliptic curves over $\FF_q$ within a given isogeny class by using the following properties. 
In case when two elliptic curves are given by~$E : \,y^{2}=x^{3}+ax+b$
 and
 $E\acute{}:\,y^{2}=x^{3}+a\acute{}x+b\acute{},$ then $E\cong E\acute{}$ over $\FF_q$ if and only if the the following relations on the coefficients hold: $a\acute{}=ac^{4}, b\acute{}=bc^{6}$ for a some $c\in\FF_q$.  

\end{remark}

\begin{example}
We give examples of maximal  and minimal elliptic curves
over finite fields over $\FF_{q}$ with  discriminant $-19$ for all $q <1000.$\\
\begin{tabular}{|c|l|l|}
\hline
${q}$& Maximal & Minimal \\
\hline
${47}$& $y^2=x^3+x+38$& $y^2=x^3+32x+27$\\[2pt]
${61}$& $y^2=x^3+6x+29$& $y^2=x^3+32x+57$\\[2pt]
${137}$& $y^2=x^3+x+36$& $y^2=x^3+61x+47$\\[2pt]
${277}$& $y^2=x^3+2x+61$& $y^2=x^3+61x+47$\\[2pt]
${311}$& $y^2=x^3+x+50$& $y^2=x^3+18x+308$\\[2pt]
${347}$& $y^2=x^3+2x+96$& $y^2=x^3+174x+12$\\[2pt]
${467}$& $y^2=x^3+2x+361$& $y^2=x^3+234x+337$\\[2pt]
${557}$& $y^2=x^3+3x+132$& $y^2=x^3+140x+295$\\[2pt]
${761}$& $y^2=x^3+x+82$& $y^2=x^3+592x+454$\\[2pt]
${997}$& $y^2=x^3+6x+493$& $y^2=x^3+500x+934$\\[2pt]
\hline
\end{tabular}
\end{example}

\subsection{Optimal Curves of Genus $2$} 
We start with a proposition which was proven in \cite{AZ}. 
\begin{proposition}\label{genustwo19}
Up to an isomorphism over the field ${\FF}_{q}$
there exists exactly one maximal  (resp.\ minimal) optimal curve
$C$ of genus $2$ over ${\FF}_{q}$, viz., the fibered product over ${\PP}^1$
of the two maximal (resp.\ minimal) optimal elliptic curves
$$
E_1: y^2=f(x) \quad {\rm   and} \quad  E_2: y^2=f(x)(\alpha\, x+\beta).
$$
\end{proposition}

\begin{example}
Here we produce examples of elliptic curves $E_2$ from the proposition above and maximal curve of genus $2$
over the finite field $\FF_{q}$ of the discriminant $-19$ and $q<1000$.\\
\begin{tabular}{|c|l|l|}
\hline
${q}$& Maximal elliptic curve& Maximal curve of genus two\\
\hline
${47}$&$y^2=(x^3+x+38)(x+30)$ & $z^2=x^6+4x^4+22x^2+33$\\[2pt]
${61}$& $y^2=(x^3+6x+29)(x+2)$ & $z^2=x^6+55x^4+18x^2+9$\\[2pt]
${137}$& $y^2=(x^3+x+36)(x+18)$& $z^2=x^6+83x^4+14x^2+77$\\[2pt]
${277}$&$y^2=(x^3+2x+61)(2\,x+80)$ & $z^2=104x^6+247x^4+185x^2+245$\\[2pt]
${311}$&$y^2=(x^3+x+50)(x+134)$& $z^2=x^6+220x^4+66x^2+19$\\[2pt]
${347}$& $y^2=(x^3+2x+96)(x+166)$ & $ z^2=x^6+196x^4+84x^2+316$\\[2pt]
${467}$&$y^2=(x^3+2x+361)(x+47)$& $z^2=x^6+326x^4+91x^2+118$\\[2pt]
${557}$&$y^2=(x^3+3x+132)(2x+266)$& $z^2=209x^6+318x^4+356x^2+421$\\[2pt]
${761}$&$y^2=(x^3+3x+132)(x+257)$& $z^2=x^6+751x^4+288x^2+98$\\[2pt]
${997}$&$y^2=(x^3+3x+132)(x+760)$& $z^2=x^6+711x^4+20x^2+30$\\[2pt]
\hline
\end{tabular}

Note that the corresponding minimal curves of genus $2$ can be obtained by twisting of maximal curves. 

\end{example}

\end{section}


\begin{section}{A degree of a projection}
We can  calculate  the degree of the maps $C \to E$, obtained via the embedding of $C$ into ${\rm  Jac}(C)\cong E^g$ and projections onto $E$. 

The following result can be found in \cite{AZ}, we include it here with the proof for the sake of completeness.
Note that proof relies on the fact that the hermitian lattice corresponding to $\rm Jac(C)$ is a free $\cO_K$-module, which holds in the case when $\FF_q$ has discriminant $-19$. 

\begin{proposition}\label{proj}
Let $C$ be an optimal curve over $\FF_q$. 
Fix an isomorphism ${ \rm  Jac}(C)\cong E^g$ such that the theta 
divisor corresponds to the hermitian form $(h_{ij})$ on $\cO_{K}^g$
on the canonical lift of ${\rm  Jac}(C)$. 
Then degree of the $k$-th projection 
\begin{equation*}
f_{k}:C\hookrightarrow {\rm  Jac}(C)\cong E^g \stackrel{pr_k}{\longrightarrow} E
\end{equation*}
equals $\det( h_{ij})_{i,j \neq k}$.
\end{proposition}
\begin{proof}
We denote the abelian variety $E^g$ by
$E_{1}\times \ldots \times E_{g}$, where $E_{i}=E$,
and consider the first projection. 
The degree of the map $f_1$ equals the intersection 
number $[C]\cdot[E_2 \times \ldots \times E_{g}]$. 
The cohomology class $[C]$ of $C$ in an appropriate cohomology theory
is $[\Theta^{g-1}/(g-1)!]$. Recall that
if $L$ is a line bundle on an abelian variety $A$ of dimension $g$ 
then by the Riemann-Roch theorem one has $(L^g/g!)^2=\deg(\varphi_{L})$,
and  $\deg(\varphi_{L})=\det(r_{ij})^2$, where the matrix
$(r_{ij})$ gives the hermitian form corresponding to the first Chern class
of the line bundle $L$.
Since the hermitian form $(h_{ij})_{i,j \neq 1}$ corresponds 
to the line bundle $\Theta|_{E_2 \times \ldots \times E_{g}}$
on the abelian variety $E_2 \times \ldots \times E_{g}$ the degree of 
$f_1$ is given by  
$$
[C]\cdot[E_2 \times \ldots \times E_{g}]=
\frac{1}{(g-1)!}(\Theta|_{E_2 \times \ldots \times E_{g}})^{g-1}=
\det((h_{ij})_{i,j \neq 1}).
$$
\end{proof}
\end{section}
\begin{section}{Properties the Automorphism Group}

In this section we prove that an optimal curve of genus $3$ over a finite field with the discriminant is $-19$ is not hyperelliptic.
Furthermore, we prove that there exists either a maximal or a minimal curve.

From the table of classification of hermitian modules with discriminant $-19$  along with 
the lemma \ref{nonhyperelliptic} proved that an order of an automorphism group of an optimal curve of genus $3$ over a finite field 
with the discriminant $-19$ is $6$.

\begin{proposition}
There exists an optimal curve $C$ of
genus $3$ over $\FF_{q}$,  namely the double covering of a  
maximal or minimal elliptic curve respectively.
\end{proposition}

\begin{proof}
The equivalence of categories as described in the Introduction tells us that a polarization of the Jacobian corresponds to a class of irreducble unimodular hermitian forms. 
According to the classification \cite{AS} of unimodular hermitian modules, there is a unique class of irreducible unimodular hermitian forms. This class can be represented by the unimodular hermitian matrix below.

$$
\left(
\begin{array}{ccc}
2 & 1 &-1 \\
1 & 3 & \frac{-3+\sqrt{-19}}{2}\\
-1 & \frac{-3-\sqrt{-19}}{2}&3\\
\end{array}
\right).
$$
Therefore by the Theorem of Oort and Ueno \cite{FO}, there exists a unique $\FF_q$-isomorphism class of optimal curves over $\FF_q$. By Proposition \ref{proj} the degree of  $f_{1} : C \to E$ is equal to the determinant

$$
{\rm det}\left(
\begin{array}{cc}
 3 & \frac{-3+\sqrt{-19}}{2}\\
 \frac{-3-\sqrt{-19}}{2}&3\\
\end{array}
\right)
$$
which is $2$. Hence $C$ is a double covering of an optimal
elliptic curve, as desired. 
\end{proof}

Now we show an optimal curve of genus $3$ is not hyperelliptic. 
\begin{lemma}{\label{nonhyperelliptic}}
Let $C$ be an optimal curve of genus $3$ over a finite field $\FF_q$ with discriminant $-19$. Then $C$ is non-hyperelliptic.
\end{lemma}

\begin{proof}
For the sake of contradiction suppose that $C$ is a hyperelliptic curve. Then there are two involutions, the first involution $\tau$ is the hyperelliptic involution and the second involution $\sigma$ corresponds to the double cover $f_1: C \rightarrow E$ from the previous proposition. So $C/\langle \sigma \rangle$ is an optimal elliptic
curve and $C/\langle \tau \rangle$ is a projective line. 
The subgroup $\langle \sigma, \tau \rangle$ is isomorphic to $\ZZ/2\ZZ \times \ZZ/2\ZZ $ and 
we have the following diagram of coverings
$$
\xymatrix{
 &\ar[ld]_{2:1}C\ar[d]^{2:1}\ar[rd]^{2:1}&\\
 C/\langle \sigma \rangle \cong E\ar[rd]^{2:1}& C/\langle \sigma  \tau \rangle \ar[d]^{2:1} & \PP^1 \ar[ld]_{2:1} \\
&\PP^1&\\}
$$
Furthermore the formal relation of groups
$$
2\cdot\frac{1}{4} \{id, \tau, \sigma, \sigma \tau \} +\{id \}=\frac{1}{2}\{id, \sigma\}+\frac{1}{2}\{id, \tau\}+\frac{1}{2}\{id, \sigma \tau\}
$$
implies the relation between idempotents in ${\rm End}({\rm Jac}(C))$
(see  \cite{KR}) and therefore we have an isogeny
\begin{equation}
{\rm Jac}(C) \sim {\rm Jac}(C/\langle \sigma \rangle) \times {\rm Jac}(C\langle \sigma \circ \tau \rangle).
\end{equation}

From the isogeny above and Hurwitz' formula,  it follows that $C \rightarrow C/\langle \sigma \rangle$ is an unramified 
double covering. Therefore the number of rational points $\#C(\FF_q)$ is even. On other hand
$\#C(\FF_q)=q+1 \pm 3m$ is odd since $m$ is odd. 
\end{proof}



Next lemma shows that a given finite field $\FF_q$ with discriminant $-19$ cannot admit minimal and  maximal curves simultaneously.

\begin{lemma} 
Let $\FF_q$ be a finite field with discriminant $-19$. Then $\FF_q$ cannot admit minimal and  maximal curves simultaneously. 

\end{lemma}
\begin{proof}

Suppose there exist a maximal curve $C_{M}$ and a minimal curve $C_{m}$ over  $\FF_q$. Then 
${\rm Jac}(C_M\times_{\FF_{q}}\FF_{q^{2}})\cong {\rm Jac}(C_m\times_{\FF_{q}}\FF_{q^{2}})$ and hence we have an $\FF_{q^2}$-isomorphism $(C_M\times_{\FF_{q}}\FF_{q^{2}})\cong (C_m\times_{\FF_{q}}\FF_{q^{2}})$. 

We denote $C_M\times_{\FF_{q}}\FF_{q^{2}}$ by $C$.
Then there are automorphisms $F_M$ and $F_m$ on $C$ which are induced by corresponding Frobenius endomorphisms. In other words 
if $\FF_q(C_M) \cong \FF_q(x,y)$ and $\FF_q(C_m)\cong\FF_q(u,w)\subset \FF_{q^{2}}(C)$ then $\FF_{q^{2}}(C)=\FF_q(x,y)$,

\begin{displaymath}
F_M:
\ \left\{ \begin{array}{ll}
\FF_{q^{2}}(C)\longrightarrow \FF_{q^{2}}(C) &\\ 
\frac{\sum \alpha_{ij}x^{i}y^{j}}{\sum \beta_{lm}x^{l}y^{m}}\longmapsto \frac{\sum \alpha_{ij}^{q}x^{i}y^{j}}{\sum \beta_{lm}^{q}x^{l}y^{m}},
\end{array}\right.
\end{displaymath}
and

\begin{displaymath}
F_m:
\ \left\{ \begin{array}{ll}
\FF_{q^{2}}(C)=\FF_{q^{2}}(u,w)\longrightarrow \FF_{q^{2}}(u,w) &\\ 
\frac{\sum \alpha_{ij}u^{i}w^{j}}{\sum \beta_{lm}u^{l}w^{m}}\longmapsto \frac{\sum \alpha_{ij}^{q}u^{i}w^{j}}{\sum \beta_{lm}^{q}u^{l}w^{m}},
\end{array}\right.
\end{displaymath}

From the  construction of the automorphisms $F_M, F_m$ it follows that
the quotient curves $C/{\la F_M\ra}$ and $C/{\la F_m \ra}$ are defined over $\FF_q$ and
$$\FF_{q}(C/{\la F_M\ra})=\FF_{q^{2}}(C)^{\la F_M\ra}=\FF_q(C_M),$$
$$\FF_{q}(C/{\la F_m\ra})=\FF_{q^{2}}(C)^{\la F_m\ra}=\FF_q(C_m).$$

The automorphisms $F_m$ and $F_M$ induce automorphisms on ${\rm Jac}(C)$ which we, by abuse of notation,  denote by $F_m$ and $F_M$, respectively.
In ${\rm End_{\FF_{q^2}}({\rm Jac}(C))}$ we have the relation $F_m^2=F_{M}^2$  and hence $F_m=-F_{M}$, since the two are distinct. On the other hand the automorphism $F_m$ and $F_M$ induce two different 
automorphisms of $C$. Therefore Torelli's Theorem $C$ must be a hyperelliptic curve. But we showed that this is impossible in Lemma \ref{nonhyperelliptic}.  
\end{proof}

\end{section}

\begin{section}{Equations of Optimal Curves of Genus $3$}
In this section we combine the theoretical results which we derived in order to produce optimal curves of genus $3$ over finite fields with discriminant $-19$.

\begin{theorem}\label{equations}
Let $C$ be an optimal curve over $\FF_q$. Then $C$ can be written in one of the following forms:
$$
\begin{array}{l}
 \left\{ 
\begin{array}{l}
z^{2}=\alpha_0+\alpha_1x+\alpha_2x^{2}+\beta_0y,  \\
y^{2}=x^{3}+ax+b,\\
\end{array}
\right.
\\

\\

\left\{ \begin{array}{ll}
z^{2}=\alpha_0+\alpha_1x+\alpha_2x^{2}+(\beta_0+\beta_1x)y, & \\
y^{2}=x^{3}+ax+b, 
\end{array}\right.
\\

\\
\left\{ \begin{array}{ll}
z^{2}=\alpha_0+\alpha_1x+\alpha_2x^{2}+\alpha_3x^{3}+(\beta_0+\beta_1x)y, & \\
y^{2}=x^{3}+ax+b, 
\end{array}\right.
\\
\end{array}
$$
with coefficients in $\FF_q$ and the equation $y^{2}=x^{3}+ax+b$ corresponding to an optimal elliptic curve.
\end{theorem}

\begin{proof}
Let $C$ be an optimal curve of genus 3 over a finite field $\FF_q$ and let $f:C\rightarrow E$ be a double covering of $C$ with the equation $y^{2}=x^{3}+ax+b$. 
Set $D = f^{-1}(\infty')=\sum_{P|\infty'}e(P|\infty')\cdot P \in{\rm Div}(C)$, where $\infty'\in E$ lies over $\infty\in \PP^{1}$ by the action $E\rightarrow\PP^{1}$, ${\rm deg \,}D=2$. 

By Riemann-Roch Theorem
$${\rm dim \,}D={\rm deg \,}D+1-g+{\rm dim \,}(W-D)={\rm dim \,}(W-D),$$
where $W$ is a canonical divisor of the curve $C$. Consequently, $D$ is equivalent to the positive divisor $W-D_1$, where ${\rm deg \,}D_1=2$. Conclude ${\rm dim \,}D={\rm dim \,}(W-D)<{\rm dim \,}W=3$. Taking into account that $C$ is a non-hyperelliptic curve and ${\rm deg \,}D=2$, we conclude ${\rm dim \,}D=1$. 

Consider the divisor $2D$. By Clifford's Theorem
$${\rm dim \,}2D\leq 1+ \frac{1}{2}{\rm deg \,}2D.$$
Therefore, ${\rm dim \,}2D\leq 3$.

We separate the proof into three cases. 
\begin{enumerate}

\item Suppose ${\rm dim \,}2D=3$. 

Then there exist linearly independent elements $1,x,z' \in L(2D)$.
Seven elements $1,x,x^{2},y,z',(z')^{2},zx$ lie in the vector space $L(4D)$. Since ${\rm dim \,}4D=6$, then there exists relation
$$a_1z'^{2}+a_2z'+a_3z'x=a_4+a_5x+a_6x^{2}+a_7y,$$
where $a_1,a_2,a_3,a_4,a_5,a_6,a_7\in \FF_q$. Recall that $a_1\neq 0$, otherwise the equation for $z'$ over $k(x,y)$ will be of degree 1, which is a contradiction, since $[k(C):k(x,y)]=2$. Dividing both parts of the equation by $a_1$ and making the substitution $z=z'+(\frac{a_2}{a_1} + \frac{a_3}{a_1}x)/2$, we obtain the equation 
$$z^{2}=\alpha_0+\alpha_1x+\alpha_2x^{2}+\beta_0y.$$

\item Suppose ${\rm dim \,}(2D)=2$ and $D=Q_1+Q_2$, where $Q_1\neq Q_2$, $Q_1,Q_2\in C(\FF_q)$.

Then we have ${\rm dim \,}(2D+Q_1)=3$, by Riemann-Roch Theorem. 
The elements $1,x,x^{2},y,z,z^{2},xz,yz,xz^{2}\in L(4D+2Q_1)$ are linearly dependent since ${\rm dim \,}(4D+2Q_1)=8$ and  $ x\in L(2D),\, z\in L(2D+Q_1),\, y\in L(3D)$. 
Therefore,  
$$z^{2}(\alpha_0+\alpha_1x)+z(\beta_0+\beta_1x+\beta_2y)+(\gamma_0+\gamma_1x+\gamma_2x^{2}+\delta y)=0.$$
Denoting the expressions in brackets by $\varphi_1, \varphi_2, \varphi_3$ respectively, we  rewrite the expression above as 
$$z^{2}\varphi_1+z\varphi_2+\varphi_3=0.$$
Knowing that $\varphi_1 = \alpha_0+\alpha_1x\neq0$ (otherwise $v_{P_1}(x)=0$) and
 the equation above can be rewritten as 
$$(z+\frac{\varphi_2}{2\varphi_1})^{2}+\frac{\varphi_3}{\varphi_1}-\frac{\varphi_2^{2}}{4\varphi_1^{2}}=0.$$
After appropriate substitutions we get the desired equation
$$z^{2}=\alpha_0+\alpha_1x+\alpha_2x^{2}+\alpha_3x^{3}+(\beta_0+\beta_1x)y.$$

\item Suppose ${\rm dim \,}(2D)=2$ and $D=Q_1+Q_2=2Q$, where $Q_1=Q_2=Q\in C(\FF_q)$.


In order to manage this case we prove that the elements $1,x,z,y,x^{2},z^{2},xy,xz$ are linearly dependent. As a corollary of this fact we obtain the equation of the second type.

In this case the functions $x\in L(2D),\, y\in L(3D)$ have pole divisors
$(x)_\infty =4Q,\,(y)_\infty =6Q$, and 
there is a function $z\in L(2D+Q)$ such that  
$(z)_\infty =5Q$.

The element $z$ is an integral element over $\FF_q[x,y]$. 
Indeed,
either 
$$
1,x,z,y,x^{2},z^{2},xy,xz \in L(10D)
$$ 
or
$$
1, x,y,z,x^2, zx, xy, z^2, zy, x^3, zx^2, xyz, z^3 \in L(15Q)
$$
are linearly dependent and in both cases we have relations with nonzero leading coefficients at $z$.
This yields that $z$ is integral over $\FF_q[x,y]$.

It is clear that $z \not \in \FF_q(x,y)$ (otherwise $2$ divides $v_{Q}(z)=5$).

The minimal polynomial of $z$ has degree $2$ and coefficients in $\FF_q[x,y]$, since the degree of extension
$[\FF_q(C): \FF_q(x,y)]$ is $2$.
Therefore we have that
$$
z^2+\sum_{i \ge 0} a_{i} zyx^{i}+ \sum_{j \ge 0} b_{j} zx^{j} +\sum_{l \ge 0} c_{l}x^{l} +\sum_{s \ge 0} d_{s} yx^{s}=0,
$$
and hence
\begin{equation}\label{minpolynomial}
\begin{array}{c}
z^2+c_0+c_1 x+c_2 x^2+d_{0} y+b_0 z + b_{1} zx+d_1 xy =\\
=-z(b_2 x^2 + \ldots)+
zy(a_0+a_1 x + \ldots)+
(c_4 x^4 + \ldots)+
y(d_2 x^2 + \ldots).\\
\end{array}
\end{equation}
Furthermore, we have 
  \begin{itemize}
   \item{$v_Q(zx^i) =-5-4i \equiv3 $ {\rm mod} $4$}
   \item{$v_Q(zyx^j) =-5-6-4i \equiv 1 $ {\rm mod} $4$}
   \item{$v_Q(x^l) =-4l \equiv 0 $ {\rm mod} $4$}
   \item{$v_Q(yx^i) =-6-4i \equiv 2 $ {\rm mod} $4$.}
 \end{itemize}
If the right part of the equation \ref{minpolynomial} is non-zero, then we can apply the strict triangle inequality. As a consequence we get that on the one hand
$$
v_{Q}(z^2+c_0+c_1 x+c_2 x^2+d_{0} y+b_0 z + b_{1} zx+d_1 xy ) \le -11
$$
and on the other hand
$$
v_{Q}(z^2+c_0+c_1 x+c_2 x^2+d_{0} y+b_0 z + b_{1} zx+d_1 xy ) \ge -10.
$$
Therefore the right part of the equation above is zero, i.\ e.\
the elements  $1,x,z,y,x^{2},z^{2},xy,xz$ are linearly dependent.

\end{enumerate}
\end{proof}

\begin{example}

We produce examples of optimal curves over finite fields with discriminant $-19$.
It suffices to find either a maximal or  a minimal curve as their existence is mutually exclusive.

\begin{tabular}{|c||c|c|}
\hline
$q$& Maximal optimal curve & Minimal optimal curve\\
\hline
${47}$&
\begin{tabular}{c}
$y^2=x^3+x+38$, \\
$z^{2}=5+45x+30x^{2}+10y$\\ 
\end{tabular}
& -\\
\hline
${61}$&
\begin{tabular}{c}
$y^2=x^3+6x+29$,\\
$z^{2}=2+35x+48x^{2}+6y$\\
\end{tabular}
& - \\
\hline
${137}$&
\begin{tabular}{c}
$y^2=x^3+x+36$,\\
$z^{2}=3+85x+82x^{2}+45y$\\
\end{tabular}
& - \\
\hline
${277}$& 
\begin{tabular}{c}
$y^2=x^3+2x+61$,\\
$z^2=x^2+33x+212+5y$\\
\end{tabular}
& -\\
\hline
${311}$& -&
\begin{tabular}{c}
$y^2=x^{3}+x+261$,\\
$z^2=140+46x+11x^{2}+78y$\\
\end{tabular}
 \\
\hline
${347}$&-  &
\begin{tabular}{c}
$y^2=x^{3}+2x+251$,\\
$z^2=182+74x+2x^{2}+5y$\\
\end{tabular}
\\
\hline
${467}$&
\begin{tabular}{c}
$y^2=x^{3}+2x+361$,\\
$z^2=259+209x+6x^{2}+10y$\\
\end{tabular} &- \\
\hline
${557}$&-  &
\begin{tabular}{c}
$y^{2}=x^{3}+2x+151$,\\
$z^{2}=439+322x+5x^{2}+122y$ \\
\end{tabular}

\\
\hline
${761}$& 
\begin{tabular}{c}
$y^{2}=x^{3}+4x+105$,\\
$z^{2}=406+131x+3x^{2}+247y$ \\
\end{tabular} &-
\\ 
\hline

${997}$&-  &
\begin{tabular}{c}
$y^2=x^3+500x+934$,\\
$z^2=x^2+336x+564+196y$ \\
\end{tabular}

\\
\hline

\end{tabular}

\end{example}
\end{section}


\bibliographystyle{plain}
\def\cprime{$'$} \def\cprime{$'$} \def\cprime{$'$}

\end{document}